\documentclass[12pt]{amsart}

\usepackage[margin=1in]{geometry}
\usepackage{amssymb}
\usepackage{xcolor}
\usepackage[all]{xy}
\usepackage[colorlinks=true,linkcolor=violet,citecolor=violet]{hyperref}

\def\logGm{\mathbf G_{\log}}
\def\BlogGm{\mathrm B\logGm}
\def\Gm{\mathbf G_{\mathrm m}}
\def\overnorm#1{\overline#1\vphantom#1}
\def\Hom{\operatorname{Hom}}
\def\End{\operatorname{End}}

\newtheorem{theorem}{Theorem}
\newtheorem{lemma}[theorem]{Lemma}
\newtheorem{proposition}[theorem]{Proposition}
\newtheorem{corollary}[theorem]{Corollary}

\theoremstyle{remark}
\newtheorem{warning}[theorem]{Warning}

\title[Rational curves in $\logGm$]{Rational curves in the logarithmic multiplicative group}

\author[Ranganathan]{Dhruv Ranganathan}
\address{Dhruv Ranganathan \\ Department of Pure Mathematics and Mathematical Statistics\\
University of Cambridge}
\email{dr508@cam.ac.uk}

\author[Wise]{Jonathan Wise}
\address{Jonathan Wise \\ Department of Mathematics\\
University of Colorado}
\email{jonathan.wise@colorado.edu}


\begin{document}

\begin{abstract}
The logarithmic multiplicative group is a proper group object in logarithmic schemes, which morally compactifies the usual multiplicative group. We study the structure of the stacks of logarithmic maps from rational curves to this logarithmic torus, and show that in most cases, it is a product of the logarithmic torus with the space of rational curves. This gives a conceptual explanation for earlier results on the moduli spaces of logarithmic stable maps to toric varieties.
\end{abstract}

\maketitle

\section{Introduction}

A rational function on a $\mathbf P^1$ is determined up to multiplicative constant by the locations and orders of its zeroes and poles, and exists if the sum of those orders is zero. Likewise, a balanced piecewise linear function on a metric tree is determined up to addition of a constant by its slopes along unbounded edges and exists provided the sum of these slopes is zero.   We observe here that both of these facts are aspects of a single phenomenon in logarithmic geometry: \textit{the space of stable genus $0$ maps to the logarithmic torus is a product of the logarithmic torus with the space of genus $0$ curves.} See Corollaries~\ref{cor: prod} and~\ref{cor: self-maps} for the precise statements of the main results, including a treatment of unstable cases.

The logarithmic torus, $\logGm$ (to be defined precisely below), is a non-representable functor on logarithmic schemes that compactifies the algebraic torus, $\Gm$. Despite its failure to be representable, one can make sense of its tropicalization as the undivided real line $\mathbf R$, and the fiber of its tropicalization map as the multiplicative group $\Gm$.  Complete toric varieties arise by pullback along the tropicalization map $\logGm^n \to \mathbf R^n$ of subdivisions of $\mathbf R^n$ into fans. 


The spaces of genus~$0$ stable maps to a toric variety are some of the most basic objects of logarithmic Gromov-Witten theory~\cite{AC11,Che10,GS13}. Indeed, they have been studied before, for instance in~\cite{CMR14b,CS12,R15b}. In prior work on the subject, careful polyhedral arguments play a role in determining the geometry of these spaces of maps.
Such arguments are of course necessary in order to obtain the precise geometric descriptions such as those in op.\ cit., however the results presented here suggest a simple underlying principle behind those results. An auxiliary goal of this note is to demonstrate that such mildly non-representable functors can clarify the geometry of (logarithmic) schemes. For instance, the space of logarithmic stable maps to $\mathbf P^1$ with fixed contact at $n\geq 3$ marked points is a toroidal modification of $\overline{\mathcal M}_{0,n}\times \mathbf P^1$, which is a key result in~\cite{CMR14b}. We show here that the analogous space of maps to $\logGm$ is $\overline{\mathcal M}_{0,n}\times \logGm$. A logarithmic modification of the target produces a logarithmic modification of the space of maps, so this provides a clear conceptual reason for this result. Analogous statements can be extracted regarding the results of~\cite{R15b}. In Section~\ref{sec: maps-of-tori} we study maps between logarithmic tori, which serves as an underlying principle for the results of~\cite{AM14,CS12}.


We note that an incarnation of $\logGm$ exists with B. Parker's theory of exploded manifolds, as the exploded manifold attached to the `fan' in $\mathbf R$ with a single non-strictly convex cone equal to $\mathbf R$, see~\cite[Section 3]{Par12b}. 

\section{Groundwork in the tropics}

\subsection{Functors of points}\label{sec: three-functors} Logarithmic geometry facilitates the interaction between the category of schemes  and the category of convex rational polyhedral cones.  

If $\Sigma$ is a rational polyhedral cone with integral lattice $L$, let $\Sigma^\vee$ be the set of vectors in the dual lattice $L^\vee$ taking nonnegative values on $\Sigma$.  A logarithmic scheme $S$ has two important sheaves of monoids, $M_S$ and $\overnorm M_S / \mathcal O_S^\ast$.  Let $X_\Sigma$ be the affine toric variety with fan $\Sigma$, and let $\mathcal A_\Sigma$ be the stack quotient of $X_\Sigma$ by its dense torus, which is canonically equipped with a logarithmic structure.  We have two identifications:
\begin{gather}
    \Hom_{\mathbf{LogSch}}(S, X_\Sigma) = \Hom(\Sigma^\vee, \Gamma(S, M_S)) \label{eqn:toric} \\
    \Hom_{\mathbf{LogSch}}(S, \mathcal A_\Sigma) = \Hom(\Sigma^\vee, \Gamma(S, \overnorm M_S))
\end{gather}

\subsection{}  If $\overnorm\alpha$ is a section of $\overnorm M_S^{\rm gp}$ then the fiber of $M_S^{\rm gp}$ over $\overnorm\alpha$ will be denoted $\mathcal O_S^\ast(-\overnorm\alpha)$.  This is a torsor under $\mathcal O_S^\ast$ and may be completed in a unique way to an invertible sheaf $\mathcal O_S(-\overnorm\alpha)$.  If $\overnorm\alpha$ is a section of $\overnorm M_S$ then the logarithmic structure gives a $\mathcal O_S^\ast$-equivariant map $\mathcal O_S^\ast(-\overnorm\alpha) \subset M_S \to \mathcal O_S$ which extends to a homomorphism $\mathcal O_S(-\overnorm\alpha) \to \mathcal O_S$.

\subsection{The logarithmic torus} 

The logarithmic multiplicative group seems to have been introduced by Kato~\cite{KatoGm}. For any logarithmic scheme $S$ with logarithmic structure $M_S$, we define $\logGm(S)$ by~\eqref{eqn:1}:
\begin{equation} \label{eqn:1}
\logGm(S) = \Gamma(S, M_S^{\rm gp})
\end{equation}
This is a contravariant functor on logarithmic schemes.  It is not representable by a scheme or even an algebraic stack with a logarithmic structure, but it does have a logarithmically \'etale cover by the toric variety $\mathbf P^1$. In similar fashion, $\logGm^n$ admits a logarithmically \'etale cover by \textit{any} complete toric variety of dimension $n$. See~\cite[Section 2.2.7]{logpic} for a detailed treatment.

\subsection{} If $S$ has trivial logarithmic structure then $\logGm(S) = \Gm(S)$.  Since the locus in any logarithmic scheme where the logarithmic structure is trivial is an open subset, we may therefore think of $\logGm$ as at least a partial compactification of $\Gm$.  

\begin{proposition} \label{prop:p1}
The map $\mathbf P^1 \to \logGm$ given in local coordinates by $(x,y) \mapsto x^{-1} y$ is a logarithmic modification, in the sense that its base change along any map $S \to \logGm$ is a logarithmic modification.
\end{proposition}
\begin{proof}
A map $S \to \logGm$ is given by $\alpha \in \Gamma(S, M_S^{\rm gp})$.  Locally in $S$, we can represent $\alpha$ as $\beta^{-1} \gamma$ for $\beta, \gamma \in \Gamma(S, M_S^{\rm gp})$.  Then $(\beta,\gamma)$ determines a map $S \to \mathbf A^2$ (with its toric logarithmic structure) by \eqref{eqn:toric}.  Then $\alpha$ lifts to $S \to \mathbf P^1$ if and only if locally in $S$ we have $\alpha \in M_S$ or $\alpha^{-1} \in M_S$.  This is equivalent to requiring $(\beta,\gamma)$ to lift to the blowup $Y$ of $\mathbf A^2$ at the origin, which proves that $\mathbf P^1 \mathop\times_{\logGm} S = Y \mathop\times_{\mathbf A^2} S$.
\end{proof}

Since logarithmic modifications are logarithmically \'etale, the proposition shows that $\logGm$ has a logarithmically \'etale cover by a logarithmic scheme and, since that logarithmic scheme is proper, it should also be regarded as proper.  Unlike its toric compactification, $\logGm$ is also a \emph{group object} in the category of logarithmic schemes.  We emphasize that such a compactification is not possible within the category of schemes, because the only equivariant schematic compactification of $\Gm$ is $\mathbf P^1$, which does not admit a group structure.




\section{Curves in the logarithmic torus} 

\subsection{The space of maps} Let $\mathfrak M$ denote the stack of prestable genus $0$ logarithmic curves. If $X$ is a category fibered in groupoids over logarithmic schemes, we denote by $\frak M(X)$ the stack of logarithmic pre-stable maps from genus~$0$ logarithmic curves to $X$.  
\begin{equation}
\frak M(X) = \left\{ (S, C, \xi) \: \Big| \: \parbox{2in}{$C$ is a proper logarithmic genus $0$ curve over $S$ and $\xi \in X(C)$} \right\}
\end{equation}
This applies in particular to $X = \logGm$.  Since $\logGm$ has a group structure, $\frak M(X)$ is a sheaf of abelian groups over $\frak M$ in the \'etale topology.

Note that the category $\frak M(X)$ admits a tropicalization, following Section~\ref{sec: three-functors}. Specifically, given a pre-stable logarithmic map to $X$ over $S$, there is an associated diagram of tropical curves over $S^{\mathrm{trop}}$, together with a map to $X^{\mathrm{trop}}$. 

\subsection{Piecewise linear functions}

We recall from \cite[Remark~7.3]{CCUW} how sections of the characteristic monoid of a logarithmic curve give piecewise linear functions on its tropicalization.

Let $C$ be a logarithmic curve over an algebraically closed field $S$ and let $C^{\rm trop}$ be its dual graph.  Each edge $e$ of $C^{\rm trop}$ corresponds to a node where $C$ has a local equation $xy = t$ in its characteristic monoid, with $t \in M_S$.  Let $\overnorm t$ be the image of $t$ in $\overnorm M_S$.  Then we refer to $\overnorm t$ as the length of $e$.

Suppose that $\overnorm\alpha \in \overnorm M_C^{\rm gp}$.  If $v$ is a vertex of $C^{\rm trop}$ then there is a corresponding component of $C$ on which $\overnorm M_C^{\rm gp}$ is constant with value $\overnorm M_S^{\rm gp}$.  We write $\overnorm\alpha(v)$ for the constant value of $\overnorm\alpha$ on the interior of this component.  If $e$ is an edge of $C$ connected vertices $v$ and $w$ then near $e$ there is a unique representation of $\overnorm\alpha$ as $\overnorm\alpha(v) + \lambda \overnorm x$, where $\overnorm x$ is the image of $x \in M_{C,e}$ in $\overnorm M_{C,e}$ and $\lambda \in \mathbf Z$.  Then restricting to $w$ we find $\overnorm\alpha(w) = \overnorm\alpha(v) + \lambda \overnorm t$.  This allows us to think of $\overnorm\alpha$ as a piecewise linear function on $C^{\rm trop}$ with slope $\lambda$ on the edge $e$, when directed from $v$ to $w$.

\begin{proposition}
Let $C_v$ be the component of $C$ corresponding to a vertex $v$ of $C^{\rm trop}$.  Then $\mathcal O_{C_v}(\overnorm\alpha) = \mathcal O_{C_v}(\sum_{e : v \to w} \lambda_e e)$ where the sum is taken over all edges leaving $v$ and $\lambda_e$ denotes the slope of $\overnorm\alpha$ on the edge $e$.
\end{proposition}
\begin{proof}
See \cite[Proposition~2.4.1]{RSW}.
\end{proof}

\begin{proposition} \label{prop:balanced}
Suppose $C$ is a logarithmic curve over an algebraically closed field and $\overnorm\alpha \in \Gamma(C, \overnorm M_C^{\rm gp})$ lifts to $M_C^{\rm gp}$.  Then $\overnorm \alpha$ is a balanced function on $C^{\rm trop}$.
\end{proposition}
\begin{proof}
Let $\alpha \in \Gamma(C, M_C^{\rm gp})$ lift $\overnorm\alpha$.  Then $\alpha$ is, by definition, a section of $\mathcal O_C^\ast(-\overnorm\alpha)$.  Equivalently, $\overnorm\alpha$ is a nowhere vanishing section of $\mathcal O_C(-\overnorm\alpha)$.  Thus $\mathcal O_C(-\overnorm\alpha)$ is trivial and in particular has degree zero.  Restricting to a component $C_v$ of $C$, we have $\mathcal O_{C_v}(-\overnorm\alpha) \simeq \mathcal O_{C_v}(\sum_{e : v \to w} \lambda_e e)$.   This implies $\sum \lambda_e = 0$, which is the balancing condition.
\end{proof}

\subsection{Contact orders} Let $C$ be a logarithmic curve.  A map $F: C \to \logGm$ is a section of $M_C^{\rm gp}$, which in turn induces a section $\overnorm\alpha$ of $\overnorm M_C^{\rm gp}$.  We regard $\overnorm\alpha$ as a linear function on the dual graph of $C$.  The slopes of $\overnorm\alpha$ on the $n$ infinite legs of the dual graph of $C$ are locally constant in $\frak M(\logGm)$.  This gives a homomorphism $\frak M(\logGm) \to \mathbf Z^n$. The \emph{contact order} of the map is defined as  the image of $[F]$ in $\mathbf Z^n$ under this map.

\subsection{Maps up to translation} The kernel of the homomorphism $\frak M(\logGm) \to \mathbf Z^n$ consists of maps $C \to \logGm$ whose associated linear function has zero slope on the infinite legs.  But such a function is effectively a bounded balanced piecewise linear function on the complement of the infinite legs in the dual graph.  Any such balanced function $\overnorm\alpha$ is constant.  In that case, $\mathcal O_C(\overnorm\alpha)$ is the pullback of $\mathcal O_S(\overnorm\alpha)$ from the base and the map $C \to \logGm$ corresponds to a trivialization of this bundle.  Indeed, the fiber of $M_X^{\rm gp}$ over $\overnorm\alpha \in \overnorm M_X^{\rm gp}$ is $\mathcal O_X^\ast(-\overnorm\alpha)$, by definition, so a section of $M_X^{\rm gp}$ in the fiber over $\overnorm\alpha$ corresponds to a nowhere vanishing section of $\mathcal O_X(\overnorm\alpha)$.  Since $X$ is proper over $S$ with reduced and connected fibers, all sections of $\mathcal O_X(\overnorm\alpha)$ over $X$ are pulled back from sections of $\mathcal O_S(\overnorm\alpha)$.  Thus a section over $S$ of the kernel of $\frak M(\logGm) \to \mathbf Z^n$ consists of pairs $(\overnorm\alpha, \alpha)$ where $\overnorm\alpha$ is a section of $\overnorm M_S^{\rm gp}$ over $S$ and $\alpha$ is a section of $\mathcal O_S^\ast(\overnorm\alpha)$.  This shows that the kernel is isomorphic to $\logGm$.

\begin{theorem}
	Let $\mathbf Z^n_0$ be the subgroup of $\mathbf Z^n$ consisting of those $n$-tuples of integers whose sum is zero.  There is an exact sequence of sheaves (in the big \'etale site) of abelian groups over $\frak M$:
\begin{equation} \label{eqn:2}
0 \to \logGm \to \frak M(\logGm) \to \mathbf Z^n_0 \to 0
\end{equation}
and the final map is smooth.
\end{theorem}
\begin{proof}
    Note that $\frak M(\logGm) \to \mathbf Z^n$ takes values in $\mathbf Z^n_0$ because by Proposition~\ref{prop:balanced} every section of $M_X^{\rm gp}$ over a rational curve $X$ induces a \emph{balanced} section of $\overnorm M_X^{\rm gp}$, which is to say that the sum of the outgoing slopes at any vertex of the dual graph is zero.  This implies that the sum of the outgoing slopes along the infinite legs is also zero.
    
	We have already proved the left exactness in the statement of the theorem.  To conclude we must prove  that $\frak M(\logGm) \to \mathbf Z_0^n$ is a smooth surjection.  
	
	We consider the smoothness first.  Since $\mathbf Z_0^n$ is \'etale over $\frak M$, it is equivalent to demonstrate that $\frak M(\logGm)$ is smooth over $\frak M$.  Consider a first order deformation of a logarithmic curve $C \subset C'$ and a section $\alpha$ of $M_C^{\rm gp}$.  Let $\overnorm\alpha$ be the image of $\alpha$ in $\overnorm M_C^{\rm gp}$.  Then $\overnorm\alpha$ extends uniquely to $C'$ since $\overnorm M_C^{\rm gp} = \overnorm M_{C'}^{\rm gp}$ when their \'etale sites are identified.  We can view $\alpha$ as a trivialization of $\mathcal O_C(\overnorm\alpha)$ and we wish to extend this to a trivialization of $\mathcal O_{C'}(\overnorm\alpha)$.  The obstructions to doing so lie in $H^1(C, \mathcal O_C(\overnorm\alpha))$.  But $\overnorm\alpha$ is a balanced function on the tropicalization of $C$, so $\mathcal O_C(\overnorm\alpha)$ has multidegree~$0$.  As $C$ is a tree of rational curves, this implies $H^1(C, \mathcal O_C(\overnorm\alpha)) = 0$.

	To prove the surjectivity, we fix a genus~$0$ logarithmic curve $C$ with tropicalization $\Gamma$ and a vector $\sigma \in \mathbf Z^n_0$.  We can construct a unique linear function $\overnorm\alpha$ on $\Gamma$ whose slopes on the legs of $\Gamma$ are given by $\sigma$.  To lift this section to an element of $\frak M(\logGm)$ in the fiber over $\sigma$, we must give a nowhere vanishing section of $\mathcal O_C(\overnorm\alpha)$.  But $\mathcal O_C(\overnorm\alpha)$ has multidegree~$0$ and the components of $C$ are rational, so $\mathcal O_C(\overnorm\alpha) \simeq \mathcal O_C$ has a nowhere vanishing section.
\end{proof}

For a point $\Gamma$ in $\mathbf Z_0^n$, let $\mathfrak M_\Gamma(\logGm)$ be its fiber in the exact sequence above. 

\begin{corollary}
Let $\mathfrak M_{\Gamma+1}(\logGm)$ denote the moduli space of $1$-marked genus~$0$ pre-stable maps to $\logGm$ with contact order $\Gamma$ and one additional marked point of contact order~$0$.  Then evaluation at the final marked point furnishes an isomorphism $\mathfrak M_{\Gamma+1}(\logGm) \simeq \mathfrak M_{\Gamma+1} \times \logGm$.
\end{corollary}

\begin{proof}
Evaluation at the new marked point with trivial contact order splits the injection in the exact sequence above, leading to the claim.
\end{proof}

\begin{corollary} \label{cor:triv-ext}
If $n \geq 3$ then the exact sequence~\eqref{eqn:2} splits and $\mathfrak M_{0,n}(\logGm) \simeq \mathfrak M_{0,n} \times \logGm \times \mathbf Z_0^n$.
\end{corollary}
\begin{proof}
  For $n \geq 3$, let $C$ be a logarithmic curve over $S$, let $\alpha \in \Gamma(C, M_C^{\rm gp})$ give an $S$-point of $\mathfrak M_{0,n}(\logGm)$, and let $x$ be a marked point of $C$, viewed as a non-logarithmic section over $S$.  Write $\overnorm\alpha$ for the image of $\alpha$ in $\overnorm M_C^{\rm gp}$.  Then $x^{-1} \overnorm M_{C} = \mathbf N \times \overnorm M_S$, canonically, and $x^\ast \overnorm\alpha = (c(x), \overnorm\alpha(x))$ where $c$ denotes the contact order of $\alpha$ at $x$.  We can view $x^\ast \alpha$ as a nowhere vanishing section of $\mathcal O_S((\overnorm\alpha(x), c(x))) = \mathcal O_S(\overnorm\alpha(x)) \otimes N_{x/X}^{\otimes c(x)}$.  But, as $n \geq 3$, the universal tangent line $N_{x/C}$ is canonically isomorphic to the line bundle associated to a boundary divisor, so the $M_S^{\rm gp}$-torsor associated to $N_{x/C}^{\otimes c(x)}$ is canonically trivial.  Likewise the $M_S^{\rm gp}$-torsor associated to $\mathcal O_S(\overnorm\alpha)$ is canonically trivial, so $x^\ast\alpha$ is thus identified with a section of $\logGm$ and we obtain a morphism $\phi : \mathfrak M(\logGm) \to \logGm$.  It follows from the canonical isomorphism~\eqref{eqn:3}  
  \begin{equation} \label{eqn:3}
    \mathcal O_S(\overnorm\alpha(x) + \overnorm\alpha'(x)) \otimes N_{x/C}^{\otimes (c(x) + c'(x))} = \mathcal O_S(\overnorm\alpha(x)) \otimes N_{x/C}^{\otimes c(x)} \otimes \mathcal O_S(\overnorm\alpha'(x)) \otimes N_{x/C}^{\otimes c'(x)}
  \end{equation}
  that $\phi$ is a homomorphism with respect to the group structure of $\mathfrak M(\logGm)$.  It is immediate that this homomorphism splits the inclusion of $\logGm$ in $\mathfrak M(\logGm)$, and therefore that $\mathfrak M_{0,n}(\logGm) \simeq \mathfrak M \times \mathbf Z_0^n \times \logGm$.
\end{proof}

\begin{warning}
The proof of Corollary~\ref{cor:triv-ext} gives a canonical splitting of the extension~\eqref{eqn:2}, for each of the $n$ marked points of the curve.  These splittings genuinely depend on the markings and distinct markings give distinct splittings.
\end{warning}

Let $\mathcal M_\Gamma(\logGm)$ denote genus~$0$ \emph{stable} maps to $\logGm$ with fixed contact orders $\Gamma$, where stability means that if $\alpha \in \Gamma(X, M_X^{\rm gp})$ is constant on a component of $X$ then that component has at least~$3$ special points.

\begin{corollary}\label{cor: prod}
Fix a vector of $n$ contact orders and genus $0$ in the combinatorial datum $\Gamma$. We have isomorphisms:
\begin{equation*}
    \mathcal M_\Gamma(\logGm) \simeq \begin{cases} \varnothing & n \leq 1 \\ \overline{\mathcal M}_{0,n} \times \logGm & n \geq 3\end{cases}
\end{equation*}
\end{corollary}
\begin{proof}
The statement for $n \leq 1$ is immediate because there are no nonconstant linear functions on a genus~$0$ tropical curve with only one infinite leg and for $n \geq 3$ is an immediate consequence of Corollary~\ref{cor:triv-ext}. 
\end{proof}

We have not included a statement for $n = 2$ because the space `stable' maps from $2$-marked rational curves with contact orders $(r,-r)$ to $\logGm$ is the nonseparated stack $(\logGm/\Gm) \times \mathrm B \mu_r$.  The difficulty is that the unique semistable component of $\mathbf P^1$ is `contracted' by any morphism $\mathbf P^1 \to \logGm$ and it is therefore necessary to contract it in the source to obtain a reasonable parameter space.  We explain how this works in the next section.

\subsection{}

We can now make explicit the relationship between these results and those in~\cite{CMR14b,R15b}. The moduli space of logarithmic stable maps to $\mathbf P^1$ admits a morphism to the space of maps to $\logGm$, by composing the universal map with $\mathbf P^1\to\logGm$ and possibly stabilizing. The resulting map
\[
\mathcal M_\Gamma(\mathbf P^1)\to \mathcal M_\Gamma(\logGm)
\]
is easily seen to be logarithmically \'etale and birational. By Corollary~\ref{cor: prod}, the moduli space of logarithmic stable maps to $\mathbf P^1$ is a logarithmic modification of $\overline{\mathcal M}_{0,n} \times \logGm$. The analogous statement holds for any toric variety. The specific nature of this modification is determined by the map on tropicalizations, which is a subdivision, described precisely in~\cite{CMR14b,R15b}

\section{Maps between logarithmic tori}\label{sec: maps-of-tori}

The following lemma is well known, but we include a proof for completeness.

\begin{lemma} \label{lem:Gm-units}
Let $S$ be a scheme.  Then $\mathcal O_S[t,t^{-1}]^\ast = \mathcal O_S^\ast \times t^{\mathbf Z_S}$.
\end{lemma}
\begin{proof}
The assertion is straightforward to check when $S$ is integral.  Let $\alpha$ be a section of $\mathcal O_S[t,t^{-1}]^\ast$.  For each point $p$ of $S$, we have $\alpha(p) = u(p) t^{n(p)}$ for some $u \in k(p)^\ast$ and $n \in \mathbf Z$.  Let $S_n$ be the set of points $p$ of $S$ where $n(p) = n$.  It is a quick exercise to see that $S_n$ contains a constructible neighborhood of each of its points.  As valuation rings are integral, it follows that $S_n$ is also stable under generization, so each $S_n$ is open.

This implies that $n(p)$ is a constructible function on $S$.  Therefore $t^{-n}$ is a section of $t^{\mathbf Z_S}$ and $\alpha(p) t^{-n(p)} \in k(p)^\ast$ for every $p \in X$.  This implies $\alpha \in \mathcal O_S^\ast$, as required.
\end{proof}

\begin{proposition} \label{prop:end-logGm}
Let $t$ denote the identity function on $\logGm$.  Any map $S \times \logGm \to \logGm$ can be represented uniquely as $\alpha t^n$ where $\alpha$ is a section of $M_S^{\rm gp}$ and $n : S \to \mathbf Z$ is a locally constant function.
\end{proposition}
\begin{proof}
Suppose that $\beta : S \times \logGm \to \logGm$ is a map.  Let $\mathbf P^1 \to \logGm$ be the map described in Proposition~\ref{prop:p1}.  Restricting $\beta$ along this map, we obtain a section $\beta' \in \Gamma(S \times \mathbf P^1, M_{S \times \mathbf P^1}^{\rm gp})$.  We note that if $j : S \times \Gm \to S \times \mathbf P^1$ is the inclusion then $M_{S \times \mathbf P^1}^{\rm gp} = j_\ast M_{S \times \Gm}^{\rm gp}$.

Let $q : S \times \mathbf P^1 \to S$ be the projection.  We have an exact sequence~\eqref{eqn:8}:
\begin{equation} \label{eqn:8}
0 \to \mathcal O_{S \times \Gm}^\ast \to M_{S \times \Gm}^{\rm gp} \to j^{-1} q^{-1} \overnorm M_S^{\rm gp} \to 0
\end{equation}
Applying $j_\ast$ gives us an exact sequence~\eqref{eqn:11}:
\begin{equation} \label{eqn:11}
0 \to j_\ast \mathcal O_{S \times \Gm}^\ast \to M_{S \times \mathbf P^1}^{\rm gp} \to q^{-1} \overnorm M_S^{\rm gp} \to 0
\end{equation}
Pushing forward to $S$ and applying Lemma~\ref{lem:Gm-units}, we obtain the bottom row of~\eqref{eqn:9}:
\begin{equation} \label{eqn:9}
\vcenter{\xymatrix{
0 \ar[r] & \mathcal O_S^\ast \ar[r] \ar[d] & M_S^{\rm gp} \ar[d] \ar[r] & \overnorm M_S^{\rm gp} \ar@{=}[d] \ar[r] & 0 \\
0 \ar[r] &  \mathcal O_S^\ast \times t^{\mathbf Z_S} \ar[r] &  q_\ast M_{S \times \mathbf P^1}^{\rm gp} \ar[r] & \overnorm M_S^{\rm gp} \ar[r] & 0
}}
\end{equation}
It follows that there is an exact sequence~\eqref{eqn:10},
\begin{equation} \label{eqn:10}
0 \to M_S^{\rm gp} \to q_\ast M_{S \times \mathbf P^1}^{\rm gp} \to t^{\mathbf Z_S} \to 0
\end{equation}
that is split by the pullback of $t$ along the second projection $S \times \mathbf P^1 \to \logGm$.  We therefore have $q_\ast M_{S \times \mathbf P^1}^{\rm gp} = M_S^{\rm gp} \times t^{\mathbf Z_S}$.  In particular, $\beta'$ can be represented uniquely as $\alpha t^n$ where $\alpha \in M_S^{\rm gp}$ and $n : S \to \mathbf Z$ is a locally constant function.

To see that this formula actually describes $\beta$, consider $\beta \alpha^{-1} t^{-n}$.  This is now a map $S \times \logGm \to \logGm$ whose restriction to $S \times \mathbf P^1$ is trivial.  Let $f : T \to S \times \logGm$ be any map.  By Proposition~\ref{prop:p1}, $f^{-1}(S \times \mathbf P^1)$ is a logarithmic modification $p : T' \to T$ and $p^\ast f^\ast \beta = 1$ by construction.  But $\Gamma(T, M_T^{\rm gp}) \to \Gamma(T', M_{T'}^{\rm gp})$ is an injection (in fact an isomorphism) \cite[Theorem~4.4.1]{logpic}, so we conclude $f^\ast \beta = 1$, as required.
\end{proof}

\begin{corollary}
We have $\End(\logGm) = \mathbf Z$ and every $S$-morphism $\logGm \to \logGm$ is uniquely representable as the product of a translation and a homomorphism.
\end{corollary}

\begin{proof}
An endomorphism is in particular a self-map of $\logGm$, so up to translation, it is given by an $n^{\mathrm{th}}$-power map as a consequence of the proposition above. Since endomorphisms must preserve the identity, the integer $n$ is the only datum distinguishing such a map.
\end{proof}

\begin{corollary}\label{cor: self-maps}
Let $\mathcal M_{0,2}(\logGm)$ denote the stack on logarithmic schemes whose $S$-points consist of a $\logGm$-torsor $C$ on $S$ and a map $C \to \logGm$.  Then $\mathcal M_{0,2}(\logGm) \simeq \coprod_{r \in \mathbf Z} \mathrm B \mu_r$ with the understanding that $\mathrm B \mu_0 = \logGm \times \BlogGm$.
\end{corollary}
\begin{proof}
Locally in $S$, there is an isomorphism $C \simeq S \times \logGm$.  By Proposition~\ref{prop:end-logGm}, a map $f : C \to \logGm$ is therefore representable locally as $\alpha t^r : S \times \logGm \to \logGm$ with $\alpha \in M_S^{\rm gp}$ and $r \in \mathbf Z$.  It follows that $C \to \logGm$ is equivariant with respect to the map $[r] : \logGm \to \logGm$.  The locally constant function $r$ gives a decomposition $\mathcal M_{0,2}(\logGm) = \coprod X_r$.

If $r \neq 0$, the fiber of $C \to \logGm$ over the identity is therefore a torsor under $\ker \: [r] = \mu_r$.  This gives a map $X_r \simeq \mathrm B \mu_r$ sending $f : C \to \logGm$ to $(r, f^{-1}(1))$.  Conversely, given any $\mu_r$-torsor $C_0$ on $S$ we may extend $C_0$ along $\mu_r \to \logGm$ to obtain a $\logGm$-torsor $C$ along with a map $C \to \logGm$.  These operations are easily seen to be inverse to one another.

If $r = 0$ then the map to $C \to \logGm$ factors uniquely through $S$, giving a factor $\logGm$.  The choice of $C$ is parameterized by $\BlogGm$, yielding $X_0 = \logGm \times \BlogGm$.
\end{proof}

Consider the moduli space of logarithmic stable maps from two-pointed $\mathbf P^1$ to a toric variety $X$, in the class of a one-parameter subgroup. As in the previous section, by composing such maps with $X\to \logGm^n$ and stabilizing the map, we see that the moduli space of such maps is obtained from a product of copies of $\logGm$ by logarithmic modification and a root construction. This is implied by the main results of~\cite{AM14,CS12}, and again, the specific logarithmic modification is determined by the tropical subdivisions that are considered explicitly there. 

\bibliographystyle{siam} 
\bibliography{LogTorus}

\end{document}